\documentclass[10pt]{article}
\font\smallit=cmti10
\font\smalltt=cmtt10

\usepackage{amssymb,latexsym,amsmath,epsfig,amsthm} 

\makeatletter

\renewcommand\section{\@startsection {section}{1}{\z@}
{-30pt \@plus -1ex \@minus -.2ex}
{2.3ex \@plus.2ex}
{\normalfont\normalsize\bfseries}}

\renewcommand\subsection{\@startsection{subsection}{2}{\z@}
{-3.25ex\@plus -1ex \@minus -.2ex}
{1.5ex \@plus .2ex}
{\normalfont\normalsize\bfseries}}

\renewcommand{\@seccntformat}[1]{\csname the#1\endcsname. }

\makeatother

\newtheorem{theorem}{Theorem}
\newtheorem{lemma}{Lemma}
\newtheorem{conjecture}{Conjecture}
\newtheorem{proposition}{Proposition}
\newtheorem{corollary}{Corollary}
\newtheorem{example}{Example}
\newtheorem{observation}{Observation}
\newtheorem{definition}{Definition}

\usepackage[colorlinks=true,citecolor=black,linkcolor=black,urlcolor=blue]{hyperref}

\begin{document}

\begin{center}
\uppercase{\bf A new estimate on complexity of binary generalized pseudostandard words}
\vskip 20pt
{\bf Josef Florian}\\
{\smallit Department of Mathematics, Faculty of Nuclear Sciences and Physical Engineering, Czech Technical University in Prague, Czech Republic}\\
\vskip 10pt
{\bf L\!'ubom\'ira Dvo\v r\'akov\'a (born Balkov\'a)}\\
{\smallit Department of Mathematics, Faculty of Nuclear Sciences and Physical Engineering, Czech Technical University in Prague, Czech Republic}\\
{\tt lubomira.dvorakova@fjfi.cvut.cz}\\
\end{center}
\vskip 30pt

\centerline{\smallit Received: , Revised: , Accepted: , Published: } 
\vskip 30pt

\centerline{\bf Abstract}
\noindent Generalized pseudostandard words were introduced by de Luca and De Luca in~\cite{LuDeLu}. In comparison to the palindromic and pseudopalindromic closure, only little is known about the generalized pseudopalindromic closure and the associated generalized pseudostandard words. We present a~counterexample to Conjecture 43 from a~paper by Blondin Mass\'e et al.~\cite{MaPa} that estimated the complexity of binary generalized pseudostandard words as ${\mathcal C}(n) \leq 4n$ for all sufficiently large $n$. We conjecture that ${\mathcal C}(n)<6n$ for all $n \in \mathbb N$.

\pagestyle{myheadings}
\markright{\smalltt INTEGERS: 16 (2016)\hfill}
\thispagestyle{empty}
\baselineskip=12.875pt
\vskip 30pt

\section{Introduction}
This paper focuses on generalized pseudostandard words. Such words were defined by de Luca and De Luca in 2006~\cite{LuDeLu} who studied generalized standard episturmian words, called generalized pseudostandard words, by considering pseudopalindromic closure of an infinite sequence of involutory antimorphisms.
While standard episturmian and pseudostandard words have been studied intensively and a~lot of their properties are known (see for instance~\cite{BuLuZa,DrJuPi,Lu,LuDeLu}), only little has been shown so far about the generalized pseudopalindromic closure that gives rise to generalized pseudostandard words. In~\cite{LuDeLu} the authors have defined the generalized pseudostandard words and proved there that the famous Thue--Morse word is an example of such words. Jajcayov\'a et al.~\cite{JaPeSt} characterize generalized pseudostandard words in the class of generalized Thue--Morse words.
Jamet et al.~\cite{JaPaRiVu} deal with fixed points of the palindromic and pseudopalindromic closure and formulate an open problem concerning fixed points of the generalized pseudopalindromic closure. The authors of this paper provide a~necessary and sufficient condition on periodicity of binary and ternary generalized pseudostandard words in~\cite{BaFl}. The most detailed study of binary generalized pseudostandard words has been so far provided by Blondin Mass\'e et al.~\cite{MaPa}:
 \begin{itemize}
 \item A~so-called normalization is described that guarantees for generalized pseudostandard words that no pseudopalindromic prefix is missed during the construction.
 \item An effective algorithm -- the generalized Justin's formula -- for generation of generalized pseudostandard words is presented.
 \item The standard Rote words are proven to be generalized pseudostandard words and the infinite sequence of antimorphisms that generates such words is studied.
  \item A~conjecture is stated saying that the complexity of an infinite binary generalized pseudostandard word $\mathbf u$, i.e., the map ${\mathcal C}: \mathbb N \to \mathbb N$ defined by ${\mathcal C}(n)=$ the number of factors of length $n$ of the infinite word $\mathbf u$, satisfies:
$${\mathcal C}(n)\leq 4n \quad \text{for sufficiently large $n$.}$$
\end{itemize}

In this paper, we provide a~counterexample to the above conjecture by construction of a~generalized pseudostandard word satisfying ${\mathcal C}(n)>4n$ for all $n\geq 10$. We moreover show that ${\mathcal C}(n)>4.5 \ n $ for infinitely many $n \in \mathbb N$.

The work is organized as follows. In Section~\ref{sec:CoW} we introduce basics from combinatorics on words. Section~\ref{sec:palindrome} deals with the palindromic closure and summarizes known results. Similarly, Section~\ref{sec:pseudopalindrome} is devoted to the pseudopalindromic closure and its properties. In Section~\ref{sec:generalized_pseudopalindrome}, the generalized pseudopalindromic closure is defined and the normalization process is described. A~counterexample to Conjecture 43 from~\cite{MaPa} is constructed and its complexity is estimated in Section~\ref{sec:Conjecture4n}. In Section~\ref{sec:open_problems} we summarize known facts about the complexity of binary generalized pseudostandard words and state a~new conjecture: ${\mathcal C}(n) <6n$ for all $n \in \mathbb N$.

\section{Basics from combinatorics on words}\label{sec:CoW}
We restrict ourselves to the binary \emph{alphabet} $\{0,1\}$, we call $0$ and $1$ \emph{letters}. A~\emph{(finite) word} $w$ over $\{0,1\}$ is any finite binary sequence. Its length $|w|$ is the number of letters $w$ contains. The empty word -- the neutral element for concatenation of words -- is denoted by $\varepsilon$ and its length is set $|\varepsilon|=0$.
The set of all finite binary words is denoted by ${\{0,1\}}^*$.
An \emph{infinite word} $\mathbf u$ over $\{0,1\}$ is any binary infinite sequence.
The set of all infinite words is denoted $\{0,1\}^{\mathbb N}$.
A finite word $w$ is a~\emph{factor} of the infinite word $\mathbf u=u_0u_1u_2\ldots$ with $u_i \in \{0,1\}$ if there exists an index $i\geq 0$ such that $w=u_iu_{i+1}\ldots u_{i+|w|-1}$. Such an index is called an \emph{occurrence} of $w$ in $\mathbf u$. The symbol ${\mathcal L}(\mathbf u)$ is used for the set of factors of $\mathbf u$ and is called the \emph{language} of $\mathbf u$, similarly ${\mathcal L}_n(\mathbf u)$ stands for the set of factors of $\mathbf u$ of length $n$.
A~\emph{left special factor} of a~binary infinite word $\mathbf{u}$ is any factor $v$ such that both $0v$ and $1v$ are factors of $\mathbf{u}$. A~\emph{right special factor} is defined analogously. Finally, a~factor of $\mathbf{u}$ that is both right and left special is called a~\emph{bispecial}. We distinguish the following types of bispecials over $\{0,1\}$:

\begin{itemize}
\item A~\emph{weak bispecial} $w$ satisfies that only $0w1$ and $1w0$, or only $0w0$ and $1w1$ are factors of $\mathbf u$.
\item A~\emph{strong bispecial} $w$ satisfies that all $0w0$, $0w1$, $1w0$ and $1w1$ are factors of $\mathbf{u}$.
\item We do not use a~special name for bispecials that are neither weak nor strong.
\end{itemize}

Let $w \in \mathcal{L}({\mathbf u})$. A~\emph{left extension} of $w$ is any word $aw \in \mathcal{L}(\mathbf{u})$, where $a \in \{0,1\}$, and a~\emph{right extension} is defined analogously. A~\emph{bilateral extension} of $w$ is then $awb \in \mathcal{L}({\mathbf u})$, where $a,b \in \{0,1\}$. The set of left (resp. right extensions) of $w$ is denoted $\mathrm{Lext}(w)$ (resp. $\mathrm{Rext}(w)$). The \emph{(factor) complexity} of $\mathbf{u}$ is the map $\mathcal{C}_{{\mathbf u}}: {\mathbb N} \rightarrow {\mathbb N}$ defined as $$\mathcal{C}_{{\mathbf u}}(n) =\text{the number of factors of ${\mathbf u}$ of length $n$.}$$
In order to determine the complexity of an infinite word $\mathbf u$, the well-known formula for the \emph{second difference of complexity}~\cite{Cas} may be useful:
\begin{equation}\label{complexity2diff}
\Delta^2\mathcal{C}_{{\mathbf u}}(n) = \Delta\mathcal{C}_{{\mathbf u}}(n+1) - \Delta\mathcal{C}_{{\mathbf u}}(n) = \sum_{w\in \mathcal{L}_n({\mathbf u})}B(w),
\end{equation}
where $$B(w) = \#\{awb \mid a, b \in \{0,1\}, awb \in \mathcal{L}({\mathbf u})\} - \#\mathrm{Rext}(w) - \#\mathrm{Lext}(w) + 1$$
and the \emph{first difference of complexity} is defined as $\Delta\mathcal{C}_{{\mathbf u}}(n)=\mathcal{C}_{{\mathbf u}}(n+1)-\mathcal{C}_{{\mathbf u}}(n)$.

It is readily seen that for any factor of a~binary infinite word ${\mathbf u}$ the following holds:
\begin{itemize}
\item $B(w) = 1$ if and only if $w$ is a~strong bispecial.
\item $B(w) = -1$ if and only if $w$ is a~weak bispecial.
\item $B(w) = 0$ otherwise.
\end{itemize}

An infinite word $\mathbf u$ is called \emph{recurrent} if each of its factors occurs infinitely many times in $\mathbf u$. It is said to be \emph{uniformly recurrent} if for every $n \in \mathbb N$ there exists a~length $r(n)$ such that every factor of length $r(n)$ of $\mathbf u$ contains all factors of length $n$ of $\mathbf u$.
We say that an infinite word ${\mathbf u}$ is \emph{eventually periodic} if there exists $v, w \in \{0,1\}^{*}$ such that ${\mathbf u}=wv^{\omega}$, where $\omega$ denotes an infinite repetition. If $w=\varepsilon$, we call $\mathbf u$ \emph{(purely) periodic}. If $\mathbf u$ is not eventually periodic, $\mathbf u$ is said to be \emph{aperiodic}.
It is not difficult to see that if an infinite word is recurrent and eventually periodic, then it is necessarily purely periodic.
A~fundamental result of Morse and Hedlund \cite{MoHe1} states that a
word $\mathbf u$ is eventually periodic if and only if for some $n$ its
complexity is less than or equal to $n$. Infinite words of complexity $n+1$ for all $n$ are called \emph{Sturmian words}, and hence they are aperiodic words of the
smallest complexity.
Among Sturmian words we distinguish the class of \emph{standard (or characteristic) Sturmian words} satisfying that their left special factors are their prefixes at the same time. The Fibonacci word from Example~\ref{Fibonacci} is a~standard Sturmian word.
The first systematic study of Sturmian words was by Morse and Hedlund in~\cite{MorHed1940}.

A \emph{morphism} is a~map $\varphi: \{0,1\}^* \rightarrow \{0,1\}^*$ such that for every $v, w \in \{0,1\}^*$ we have $\varphi(vw) = \varphi(v) \varphi(w)$. It is clear that in order to define a~morphism, it suffices to provide letter images. A morphism is
\emph{prolongable} on $a \in \{0,1\}$ if $|\varphi(a)|\geq 2$ and
$a$ is a prefix of $\varphi(a)$. 
If $\varphi $ is prolongable on
$a$, then $\varphi^n(a)$ is a proper prefix of $\varphi^{n+1}(a)$
for all $n \in \mathbb{N}$. Therefore, the sequence
$(\varphi^n(a))_{n\geq 0}$ of words defines an infinite word $\mathbf u$
that is a~fixed point of $\varphi$. Such a word $\mathbf u$ is a~\emph{(pure) morphic word}.

\begin{example}\label{Fibonacci}
The most studied Sturmian word is the
so-called Fibonacci word
\[{\mathbf u}_F=01001010010010100101001001010010\ldots\]
fixed by the morphism $\varphi_F(0)=01$ and $\varphi_F(1)=0$.
\end{example}

\begin{example}\label{ThueMorse}
Another well-known morphic word that however does not belong to Sturmian words is the
Thue-Morse word
\[{\mathbf u}_{TM}=01101001100101101001011001101001\ldots\]
fixed by the morphism $\varphi_{TM}(0)=01$ and $\varphi_{TM}(1)=10$ (we start with the letter $0$ when generating ${\mathbf u}_{TM}$).
\end{example}

An \emph{involutory antimorphism} is a~map $\vartheta: \{0,1\}^* \rightarrow \{0,1\}^*$ such that for every $v, w \in \{0,1\}^*$ we have $\vartheta(vw) = \vartheta(w) \vartheta(v)$ and moreover $\vartheta^2$ equals identity. There are only two involutory antimorphisms over the alphabet $\{0,1\}$: the \emph{reversal (mirror) map} $R$ satisfying $R(0)=0, R(1)=1$, and the \emph{exchange antimorphism} $E$ given by $E(0)=1, E(1)=0$. We use the notation $\overline{0} = 1$ and $\overline{1} = 0$, $\overline{E} = R$ and $\overline{R} = E$. A finite word $w$ is a~\emph{palindrome} if $w = R(w)$, and $w$ is an $E$-\emph{palindrome} ({\em pseudopalindrome}) if $w = E(w)$.

\section{Palindromic closure}\label{sec:palindrome}
In this section we describe the construction of binary infinite words generated by the palindromic closure. Further on, we recall some properties of such infinite words. We use the papers~\cite{DrJuPi,Lu} as our source.
\begin{definition}
Let $w \in \{0,1\}^*$. The \emph{palindromic closure} $w^R$ of a~word $w$ is the shortest palindrome having $w$ as prefix.
\end{definition}

Consider for instance the word $w=0100$. Its palindromic closure $w^R$ equals $010010$. It is readily seen that $|w| \leq |w^R| \leq 2|w|-1$. For $w = 010$ we have $w^R = 010$ and for $w = 0001$ we obtain $w^R = 0001000$. It is worth noticing that the palindromic closure can be constructed in the following way: Find the longest palindromic suffix $s$ of $w$. Denote $w = ps$. Then $w^R = psR(p)$. For instance, for $w = 0100$ we have $s = 00$ and $p = 01$. Thus $w^R = 01\underline{00}10$.

\begin{definition}
Let $\Delta = \delta_1 \delta_2 \ldots$, where $\delta_i \in \{0,1\}$ for all $i \in \mathbb{N}$. The infinite word $\mathbf{u}(\Delta)$ \emph{generated by the palindromic closure (or $R$-standard word)} is the word whose prefixes $w_n$ are obtained from the recurrence relation
$$w_{n+1} = (w_n \delta_{n+1})^{R},$$ $$w_0 = \varepsilon.$$ The sequence $\Delta$ is called the \emph{directive sequence} of the word $\mathbf{u}(\Delta)$.
\end{definition}

\noindent {\bf Properties of the $R$-standard word} $\mathbf{u}=\mathbf{u}(\Delta)\in \{0,1\}^{\mathbb N}$:
\begin{enumerate}
\item The sequence of prefixes $(w_k)_{k\geq 0}$ of $\mathbf{u}$ contains every palindromic prefix of~$\mathbf{u}$.
\item The language of $\mathbf u$ is closed under reversal, i.e., $w$ is a~factor of $\mathbf u$ $\Leftrightarrow$ $R(w)$ is a~factor of $\mathbf u$.
\item The word $\mathbf u$ is uniformly recurrent.
\item Every left special factor of $\mathbf u$ is a~prefix of $\mathbf u$.
\item If $w$ is a bispecial factor of $\mathbf u$, then $w=w_k$ for some $k$.
\item Since $\mathbf u$ is (uniformly) recurrent, it is either aperiodic or purely periodic.
\item The word $\mathbf u$ is standard Sturmian if and only if both $0$ and $1$ occur in the directive sequence $\Delta$ infinitely many times.
\item The word $\mathbf u$ is periodic if and only if $\Delta$ is of the form $v0^{\omega}$ or $v1^{\omega}$ for some $v \in \{0,1\}^*$.
\end{enumerate}

\begin{example}
The Fibonacci word $\mathbf{u_F}$ defined in Example~\ref{Fibonacci} is the most famous example of an infinite word generated by the palindromic closure. It is left as an exercise for the reader to show that $\mathbf{u_F} = \mathbf{u}((01)^{\omega})$. Let us form the first few prefixes $w_k$:

\begin{align}
w_1 =& \; 0 \nonumber \\
w_2 =& \; 010 \nonumber \\
w_3 =& \; 010010 \nonumber \\
w_4 =& \; 01001010010. \nonumber
\end{align}

\end{example}
\section{Pseudopalindromic closure}\label{sec:pseudopalindrome}
Let us recall here the definition of the pseudopalindromic closure and the construction of binary infinite words generated by the pseudopalindromic closure. Some of their properties are similar as for the palindromic closure, but in particular their complexity is already slightly more complicated. Pseudopalindromes and the pseudopalindromic closure have been studied for instance in~\cite{BuLuZa, LuDeLu}.

\begin{definition}
Let $w \in \{0,1\}^*$. The \emph{pseudopalindromic closure} $w^E$ of a~word $w$ is the shortest $E$-palindrome having $w$ as prefix.
\end{definition}

Consider $w=0010$ which has pseudopalindromic closure $w^E=001011$. The following inequalities hold: $|w| \leq |w^E| \leq 2|w|$. For instance for $w = 0101$ we have $w^E = 0101$, while for $w = 000$ we get $w^E = 000111$. Let us point out that the pseudopalindromic closure may be constructed in the following way: Find the longest pseudopalindromic suffix of~$w$. Denote it by $s$ and denote the remaining prefix by $p$, i.e., $w = ps$. Then $w^E = psE(p)$. For $w=0010$, we obtain $p = 00$ and $s = 10$, therefore $w^E = 00\underline{10}11$.

\begin{definition}
Let $\Delta = \delta_1 \delta_2 \ldots$, where $\delta_i \in \{0,1\}$ for all $i \in \mathbb{N}$. The infinite word $\mathbf{u}_E(\Delta)$ \emph{generated by the pseudopalindromic closure (or $E$-standard or pseudostandard word)} is the word whose prefixes $w_n$ are obtained from the recurrence relation
$$w_{n+1} = (w_n \delta_{n+1})^{E},$$
$$w_0 = \varepsilon.$$
The sequence $\Delta$ is called the \emph{directive sequence} of the word $\mathbf{u}_E(\Delta)$.
\end{definition}

\noindent {\bf Properties of the $E$-standard word} $\mathbf{u}=\mathbf{u}_E(\Delta) \in \{0,1\}^{\mathbb N}$:
\begin{enumerate}
\item The sequence of prefixes $(w_k)_{k\geq 0}$ of $\mathbf{u}$ contains every pseudopalindromic prefix of $\mathbf{u}$.
\item The language of $\mathbf u$ is closed under the exchange antimorphism, i.e., $w$ is a~factor of $\mathbf u$ $\Leftrightarrow$ $E(w)$ is a~factor of $\mathbf u$.
\item The word $\mathbf u$ is uniformly recurrent.
\item A close relation between $R$-standard and $E$-standard words has been revealed in Theorem 7.1 in~\cite{LuDeLu}:
Let $\Delta = \delta_1 \delta_2 \ldots$, where $\delta_i \in \{0,1\}$ for all $i \in \mathbb{N}$. Then $$\mathbf{u}_E(\Delta)=\varphi_{TM}({\mathbf u}(\Delta)).$$
In words, any $E$-standard word is the image by the Thue-Morse morphism $\varphi_{TM}$ of the $R$-standard word with the same directive sequence $\Delta$. Moreover, the set of pseudopalindromic prefixes of $\mathbf{u}_E(\Delta)$ equals the image by $\varphi_{TM}$ of the set of palindromic prefixes of $\mathbf{u}(\Delta)$.
\item If $\Delta$ contains both $0$ and $1$ infinitely many times, then every prefix of $\mathbf u$ is left special.
\item In contrast to infinite words generated by the palindromic closure, $\mathbf u$ can contain left special factors that are not prefixes. Nevertheless, such left special factors can be of length at most $2$.
\item If $w$ is a bispecial factor of $\mathbf u$ of length at least $3$, then $w=w_k$ for some $k$.
\item Since $\mathbf u$ is (uniformly) recurrent, it is either aperiodic or purely periodic.
\item The complexity of $\mathbf u$ satisfies ${\mathcal C}_{\mathbf u}(n+1)-{\mathcal C}_{\mathbf u}(n)=1$ for all $n \geq 3$ if and only if both $0$ and $1$ occur in the directive sequence $\Delta$ infinitely many times.
\item The word $\mathbf u$ is periodic if and only if $\Delta$ is of the form $v0^{\omega}$ or $v1^{\omega}$ for some $v \in \{0,1\}^*$.
\end{enumerate}

\begin{example}
Let us illustrate the construction of an infinite word generated by the pseudopalindromic closure for $\mathbf{u} = \mathbf{u}_E((01)^{\omega})$. Here are the first prefixes $w_k$:
\begin{align}
w_1 =& \; 01 \nonumber \\
w_2 =& \; 011001 \nonumber \\
w_3 =& \; 011001011001 \nonumber \\
w_4 =& \; 0110010110011001011001. \nonumber
\end{align}
Notice that $1$ and $10$ are left special factors that are not prefixes.
The reader can also check that $\mathbf u$ is the image by $\varphi_{TM}$ of the Fibonacci word, i.e., $\mathbf u=\varphi_{TM}({\mathbf u}_F)$.
\end{example}
\section{Generalized pseudopalindromic closure}\label{sec:generalized_pseudopalindrome}
Generalized pseudostandard words form a~generalization of infinite words generated by the palindromic  (resp. pseudopalindromic) closure; such a~construction was first described and studied in~\cite{LuDeLu}.
Let us start with their definition and known properties; we use the papers~\cite{JaPaRiVu,LuDeLu,MaPa}.

\subsection{Definition of generalized pseudostandard words}
Let us underline that we again restrict ourselves only to the binary alphabet $\{0,1\}$.
%
%
\begin{definition}
Let $\Delta = \delta_1 \delta_2 \ldots$ and $\Theta = \vartheta_1 \vartheta_2 \ldots$, where $\delta_i \in \{0,1\}$ and $\vartheta_i \in \{E, R\}$ for all $i \in \mathbb{N}$. The infinite word $\mathbf{u}(\Delta, \Theta)$ \emph{generated by the generalized pseudopalindromic closure (or generalized pseudostandard word)} is the word whose prefixes $w_n$ are obtained from the recurrence relation
$$w_{n+1} = (w_n \delta_{n+1})^{\vartheta_{n+1}},$$ $$w_0 = \varepsilon.$$ The sequence $\Lambda = (\Delta, \Theta)$ is called the \emph{directive bi-sequence} of the word $\mathbf{u}(\Delta, \Theta)$.
\end{definition}

\noindent {\bf Properties of the generalized pseudostandard word} $\mathbf u=\mathbf{u}(\Delta, \Theta) \in \{0,1\}^{\mathbb N}$:
\begin{enumerate}
\item If $R$ (resp. $E$) is contained in $\Theta$ infinitely many times, then the language of $\mathbf u$ is closed under reversal (resp. under the exchange antimorphism).
\item The word $\mathbf u$ is uniformly recurrent.
\end{enumerate}

\subsection{Normalization}
In contrast to $E$- and $R$-standard words, the sequence $(w_k)_{k\geq 0}$ of prefixes of a~generalized pseudostandard word ${\mathbf u}(\Delta, \Theta)$ does not have to contain all $E$-palindromic and palindromic prefixes of ${\mathbf u}(\Delta, \Theta)$. Blondin Mass\'e et al.~\cite{MaPa} introduced the notion of normalization of the directive bi-sequence.

\begin{definition}
A~directive bi-sequence $\Lambda=(\Delta, \Theta)$ of a~generalized pseudostandard word $\mathbf{u}(\Delta, \Theta)$ is called \emph{normalized} if the sequence of prefixes $(w_k)_{k\geq 0}$ of $\mathbf{u}(\Delta, \Theta)$ contains all $E$-palindromic and palindromic prefixes of ${\mathbf u}(\Delta, \Theta)$.
\end{definition}

\begin{example} \label{ex:norm}
Let $\Lambda=(\Delta, \Theta) = ((011)^{\omega}, (EER)^{\omega})$. Let us write down the first prefixes of $\mathbf{u}(\Delta, \Theta)$:
\begin{align}
w_1 =& \;01 \nonumber \\
w_2 =& \;011001 \nonumber \\
w_3 =& \;01100110 \nonumber \\
w_4 =& \;0110011001. \nonumber
\end{align}
The sequence $w_k$ does not contain for instance the palindromic prefixes $0$ and $0110$ of $\mathbf{u}(\Delta, \Theta)$.
\end{example}
The authors of~\cite{MaPa} proved that every directive bi-sequence $\Lambda$ can be normalized, i.e., transformed to such a~form $\widetilde \Lambda$ that the new sequence $(\widetilde{w_k})_{k\geq 0}$ contains already every $E$-palindromic and palindromic prefix and $\widetilde \Lambda$ generates the same generalized pseudostandard word as~$\Lambda$.

\begin{theorem}[\cite{MaPa}]\label{thm:norm}
Let $\Lambda = (\Delta, \Theta)$ be a~directive bi-sequence. Then there exists a~normalized directive bi-sequence $\widetilde{\Lambda} = (\widetilde{\Delta}, \widetilde{\Theta})$ such that ${\mathbf u}(\Delta, \Theta) = {\mathbf u}(\widetilde{\Delta}, \widetilde{\Theta})$.

Moreover, in order to normalize the sequence $\Lambda$, it suffices firstly to execute the following changes of its prefix (if it is of the corresponding form):
\begin{itemize}
\item $(a\bar{a}, RR) \rightarrow (a\bar{a}a, RER)$,
\item $(a^i, R^{i-1}E) \rightarrow (a^i\bar{a}, R^iE)$ for $i \geq 1$,
\item $(a^i\bar{a}\bar{a}, R^iEE) \rightarrow (a^i\bar{a}\bar{a}a, R^iERE)$ for $i \geq 1$,
\end{itemize}
and secondly to replace step by step from left to right every factor of the form:
\begin{itemize}
\item $(ab\bar{b}, \vartheta\overline{\vartheta}\overline{\vartheta}) \rightarrow (ab\bar{b}b, \vartheta\overline{\vartheta}\vartheta\overline{\vartheta})$,
\end{itemize}
where $a, b \in \{0,1\}$ and $\vartheta \in \{E,R\}$.
\end{theorem}

\begin{example} \label{ex:norm2}
Let us normalize the directive bi-sequence $\Lambda = ((011)^{\omega}, (EER)^{\omega})$ from Example~\ref{ex:norm}.
According to the procedure from Theorem~\ref{thm:norm}, we transform first the prefix of $\Lambda$. We replace $(0,E)$ with $(01,RE)$ and get $\Lambda_1 = (01(110)^{\omega}, RE(ERE)^{\omega})$. The prefix of $\Lambda_1$ is still of a~forbidden form, we replace thus the prefix $(011,REE)$ with $(0110, RERE)$ and get $\Lambda_2 = (0110(101)^{\omega}, RERE(REE)^{\omega})$. The prefix of $\Lambda_2$ is now correct. It remains to replace from left to right the factors $(101, REE)$ with $(1010, RERE)$. Finally, we obtain $\widetilde{\Lambda} = (0110(1010)^{\omega}, RERE(RERE)^{\omega})=(01(10)^{\omega}, (RE)^{\omega})$, which is already normalized. Let us write down the first prefixes $(\widetilde{w_k})_{k\geq 0}$ of ${\mathbf u}(\widetilde{\Lambda})$:
\begin{align}
\widetilde{w_1} =& \;0 \nonumber \\
\widetilde{w_2} =& \;01 \nonumber \\
\widetilde{w_3} =& \;0110 \nonumber \\
\widetilde{w_4} =& \;011001. \nonumber
\end{align}
We can notice that the new sequence $(\widetilde{w_k})_{k\geq 0}$ now contains the palindromes $0$ and $0110$ that were skipped in Example~\ref{ex:norm}.
\end{example}

\section{Conjecture 4n}\label{sec:Conjecture4n}
As a~new result, we will construct a~counterexample to Conjecture $4n$ (stated as Conjecture $43$ in~\cite{MaPa}):
\begin{conjecture}[Conjecture 4n]
For every binary generalized pseudostandard word ${\mathbf u}$ there exists $n_0 \in {\mathbb N}$ such that  $\mathcal{C}_{{\mathbf u}}(n) \leq 4n$ for all $n>n_0$.
\end{conjecture}

We have found a~counterexample ${\mathbf u_p} = {\mathbf u}(1^{\omega}, (EERR)^{\omega})$ that satisfies $\mathcal{C}_{{\mathbf u_p}}(n) > 4n$ for all $n \geq 10$. Moreover, we will show in the end of this section that $\mathbf u_p$ even satisfies ${\mathcal C}(n)\geq 4.577 \ n$ for infinitely many $n \in \mathbb N$. Let us write down the first prefixes $w_n$ of ${\mathbf u}_p$:
\begin{align}
    w_1 =& \; 10 \nonumber \\
    w_2 =& \; 1010 \nonumber \\
    w_3 =& \; 10101 \nonumber \\
    w_4 =& \; 1010110101 \nonumber \\
    w_5 =& \; 1010110101100101001010 \nonumber \\
    w_6 =& \; 1010110101100101001010110101100101001010 \nonumber
\end{align}
It is readily seen that $w_{4k+1}$ and $w_{4k+2}$ are $E$-palindromes, while $w_{4k+3}$ and $w_{4k+4}$ are palindromes for all $k \in \mathbb N$.

The aim of this section is to prove the following theorem.
\begin{theorem}\label{thm:counterexample}
The infinite word ${\mathbf u_p} = {\mathbf u}(1^{\omega}, (EERR)^{\omega})$ satisfies
$$\mathcal{C}_{{\mathbf u_p}}(n) > 4n \quad \text{for all $n \geq 10$.}$$
\end{theorem}

In order to prove Theorem~\ref{thm:counterexample}, we have to describe all weak bispecial factors and find enough strong bispecial factors so that it provides us with a~lower bound on the second difference of complexity (see Equation~\eqref{complexity2diff}) that leads to the strict lower bound equal to $4n$ on the complexity of $\mathbf u_p$. The partial steps will be formulated in several lemmas and observations.

Let us start with a~description of the relation between the consecutive prefixes $w_k$ and $w_{k+1}$ that will turn out to be useful in many proofs. The knowledge of the normalized form of the directive bi-sequence is needed.

\begin{observation}\label{obs:normtvar}
The directive bi-sequence $\Lambda = (1^{\omega}, (EERR)^{\omega})$ has the normalized form $\widetilde{\Lambda} = (1010(1)^{\omega}, RERE(RREE)^{\omega})$.
\end{observation}
\begin{proof}
The normalized form is obtained using the algorithm from Theorem~\ref{thm:norm}.
\end{proof}
\begin{example}\label{ex:normtvar}
The prefixes $\widetilde{w_n}$ of ${\mathbf u}(\widetilde{\Lambda})$ satisfy:
\begin{align}
    \widetilde{w_1} =& \; 1 \nonumber \\
    \widetilde{w_2} =& \; 10 \nonumber \\
   \widetilde{w_3} =& \; 101 \nonumber \\
   \widetilde{w_n} =& \; w_{n-2} \quad \text{for all $n\geq 4$.} \nonumber
\end{align}
\end{example}

\begin{lemma} \label{lemma:consecutive_members}
For the infinite word ${\mathbf u_p} = {\mathbf u}(1^{\omega}, (EERR)^{\omega})$ and $k \in \mathbb{N}$, the following relations hold. For $z \leq 0$ we set $w_z = \varepsilon$.
\begin{align}
    w_{4k+1} = & \;w_{4k}10E(w_{4k})  \nonumber \\
    w_{4k+2} = & \;w_{4k+1}w_{4k-2}^{-1}w_{4k+1}  \nonumber \\
    w_{4k+3} = & \;w_{4k+2}(010)^{-1}R(w_{4k+2}) \nonumber \\
    w_{4k+4} = & \;w_{4k+3}w_{4k}^{-1}w_{4k+3}. \nonumber
\end{align}
\end{lemma}
\begin{proof}
The statement follows from Theorem~29 in~\cite{MaPa}. We prefer however to prove it here also for integrity of our paper.
One can easily check that the statement holds for $w_1$, $w_2$, $w_3$ and $w_4$. Let $k\geq 1$.
\begin{itemize}
\item
In order to get the $E$-palindrome $w_{4k+1}$, it is necessary to find the longest $E$-palindromic suffix of $w_{4k}1$. In other words, it is necessary to find the longest $E$-palindromic suffix preceded by $0$ of the palindrome $w_{4k}$. Taking into account the normalized form of the directive bi-sequence $\widetilde{\Lambda}$ from Observation~\ref{obs:normtvar}, for every $E$-palindromic (resp. palindromic) prefix $p$ of $\mathbf{u_p}$ there exists $\ell \in \mathbb N$ such that $p = \widetilde{w_\ell}$. Therefore all $E$-palindromic suffixes of $w_{4k}$ are of the form $R(\widetilde{w_\ell})$, where $\widetilde{w_\ell} = E(\widetilde{w_\ell})$. However, we search only for the longest $E$-palindromic suffix of $w_{4k}$ preceded by $0$. If $0R(\widetilde{w_\ell})$ is a~suffix of $w_{4k}$, then $\widetilde{w_\ell}0$ has to be the prefix of $w_{4k}$. Using the normalized form $\widetilde{\Lambda}$ we nevertheless notice that no $\widetilde{w_\ell} = E(\widetilde{w_\ell})$ is followed by $0$. Consequently, $w_{4k+1} = w_{4k}10E(w_{4k})$.

\item
To obtain the $E$-palindrome $w_{4k+2}$, we look for the longest $E$-palindromic suffix of $w_{4k+1}1$. We proceed analogously as in the previous case, thus we search for the longest $E$-palindromic prefix $\widetilde{w_\ell}$ of $w_{4k+1}$ followed by $1$. Then $E(\widetilde{w_\ell}1) = 0\widetilde{w_\ell}$ is the longest $E$-palindromic suffix of $w_{4k+1}$ preceded by $0$. It follows from the form of $\widetilde{\Lambda}$ that every $E$-palindromic prefix $\widetilde{w_\ell}$ of $w_{4k+1}$ is followed by $1$.
Moreover, according to Example~\ref{ex:normtvar}, $E$-palindromes in the sequence $(w_k)_{k\geq 0}$ coincide with $E$-palindromes in the sequence $(\widetilde{w_k})_{k\geq 0}$, therefore the longest $E$-palindromic prefix $\widetilde{w_\ell}$ of $w_{4k + 1}$ followed by $1$ is $w_{4k-2}$. Consequently, $w_{4k+2} = w_{4k+1}w_{4k-2}^{-1}w_{4k+1}$.
\item The remaining two cases are similar. They are left as an exercise for the reader.
\end{itemize}
\end{proof}

It is not difficult to find strong bispecials among members of the sequence $(w_k)_{k\geq 0}$.
\begin{lemma}\label{lemma:strongBS}
Consider ${\mathbf u_p} = {\mathbf u}(1^{\omega}, (EERR)^{\omega})$ and let $k \in \mathbb N$. Then $w_{4k+1}$ and $w_{4k+3}$ are strong bispecials of ${\mathbf u_p}$. Moreover, $1w_{4k+1}0$ is a~central factor of $w_{4(k+1)+1}$ and $0w_{4k+3}0$ is a~central factor of $w_{4(k+1)+3}$.
\end{lemma}
\begin{proof}
Let us show the statement for the $E$-palindrome $w_{4k+1}$.
The proof for the palindrome $w_{4k+3}$ is similar.

Since $\Delta = 1^{\omega}$, the prefix $w_{4k+1}$ is followed by $1$. Consider now any $E$-palindrome $w_j$ such that $j>4k+1$. Since $w_j = E(w_j)$ and $w_{4k+1}1$ is a~prefix of $w_j$, the factor $0E(w_{4k+1})=0w_{4k+1}$ is a~suffix of $w_j$. The prefix $w_j$ is again followed by $1$, therefore $0w_{4k+1}1 \in \mathcal{L}(\mathbf{u_p})$. Consider further on any palindrome $w_\ell$ such that $\ell > 4k+1$. Since $w_{4k+1}1$ is again a~prefix of $w_\ell = R(w_\ell)$, the factor $1R(w_{4k+1})$ is a~suffix of $w_\ell$. The prefix $w_\ell$ is followed by $1$, thus $1R(w_{4k+1})1 \in \mathcal{L}(\mathbf{u_p})$. Since the language is closed under $R$ and $E$, we deduce that $1w_{4k+1}1, 0w_{4k+1}0 \in \mathcal{L}(\mathbf{u_p})$.

Let us find the missing bilateral extension $1w_{4k+1}0$ of the $E$-palindrome $w_{4k+1}$. We will show that $1w_{4k+1}0$ is a~central factor of $w_{4(k+1)+1}$. By Lemma~\ref{lemma:consecutive_members} we have $$w_{4(k+1)+1} = w_{4(k+1)}10E(w_{4(k+1)}).$$ The factor $w_{4k}1$ is a~prefix of the palindrome $w_{4(k+1)}$, therefore $1R(w_{4k}) = 1w_{4k}$ is a~suffix of $w_{4(k+1)}$. It implies moreover that $E(1w_{4k})=E(w_{4k})0$ is a~prefix of $E(w_{4(k+1)})$. Altogether we see that $$1w_{4k+1}0 = 1w_{4k}10E(w_{4k})0$$ is a~central factor of $w_{4(k+1)+1}$.
\end{proof}

Let us indicate how we managed to find weak bispecials. The factor $w_k$ has $w_{k-1}1$ as prefix. When constructing $w_k=\vartheta(w_k)$, one looks for the longest $\vartheta$-palindromic suffix of $w_{k-1}1$. In order to get a~weak bispecial, we look instead for the longest $\overline{\vartheta}$-palindromic suffix of $w_{k-1}1$. If this suffix is longer than the longest $\vartheta$-palindromic suffix, we check whether its bilateral extension is also a~$\overline{\vartheta}$-palindrome. If yes, we extend it and continue in the same way. When we arrive at the moment where it is not possible to extend it any more, we have a~bispecial factor: We get either a~factor of the form $ap\overline{a}$, where $p=R(p)$, and since the language is closed under reversal, $\overline{a}pa$ is a~factor of ${\mathbf u_p}$ too. Or we get a~factor of the form $apa$, where $p=E(p)$, and since the language is closed under the exchange antimorphism, $\overline{a}p\overline{a}$ is a~factor of ${\mathbf u_p}$ too.

\begin{example}\label{ex:shortest_weakBS}
Let us show how to obtain the shortest weak bispeacials using the above described way. The factor $w_5$ is an $E$-palindrome. The longest $E$-palindromic suffix of $w_4 1$ is $\varepsilon$, while the longest palindromic suffix is $1101011$. This palindrome may be moreover extended by $0$ to $011010110$. We have thus obtained the shortest weak bispecial. The second weak bispecial of the same length is $E(011010110)=100101001$.
\end{example}

Let us now provide a~formal description of weak bispecials.

\begin{lemma}\label{lemma:BSs_k}
Consider ${\mathbf u_p} = {\mathbf u}(1^{\omega}, (EERR)^{\omega})$.
Then for all $k \in \mathbb N, \ k \geq 1,$ the following factors of ${\mathbf u_p}$ are bispecials:
$$s_{4k+1} = R(w_{4(k-1)+1})w_{4(k-1)}^{-1}w_{4(k-1)+3}w_{4(k-1)}^{-1}w_{4(k-1)+1},$$
$$s_{4k+3} = E(w_{4(k-1)+3})w_{4k-2}^{-1}w_{4k+1}w_{4k-2}^{-1}w_{4(k-1)+3}.$$ Moreover, the palindrome $s_{4k+1}$ is contained in the prefix $w_{4k+1}$ and has the bilateral extensions $1s_{4k+1}0$ and $0s_{4k+1}1$, and the $E$-palindrome $s_{4k+3}$ is contained in the prefix $w_{4k+3}$ and has the bilateral extensions $0s_{4k+3}0$ and~$1s_{4k+3}1$.
\end{lemma}
\begin{proof}
Let us show the statement for $s_{4k+1}$. The proof for $s_{4k+3}$ is similar.
Using Lemma~\ref{lemma:consecutive_members} we can write $w_{4k+1} = w_{4k}10E(w_{4k}).$ The prefix $w_{4k}$ and the suffix $E(w_{4k})$ can be again rewritten as follows: $$w_{4k+1} = w_{4(k-1)+3}w_{4(k-1)}^{-1}w_{4(k-1)+3}10E(w_{4(k-1)+3}w_{4(k-1)}^{-1}w_{4(k-1)+3}).$$
 Thus $w_{4k+1}$ has $w_{4(k-1)+3}$ as prefix. The factor $w_{4(k-1)+3}$ has certainly $w_{4(k-1)+1}1$ as prefix. Since the factor $w_{4(k-1)+3}$ is a~palindrome, the factor $1R(w_{4(k-1)+1})$ is its suffix. Using the above form of $w_{4k+1}$, we know $1R(w_{4(k-1)+1})w_{4(k-1)}^{-1}w_{4(k-1)+3}$ is a~factor of $w_{4k+1}$.

 Thanks to Lemma~\ref{lemma:strongBS} we know that $1w_{4(k-1)+1}0$ is a~central factor of $w_{4k+1}$. Let us use again Lemma~\ref{lemma:consecutive_members} to rewrite $w_{4(k-1)+1}$: $$w_{4(k-1)+1} = w_{4(k-1)}10E(w_{4(k-1)}).$$ The already constructed factor $1R(w_{4(k-1)+1})w_{4(k-1)}^{-1}w_{4(k-1)+3}$ is therefore followed by $w_{4(k-1)}^{-1}w_{4(k-1)+1}0$. Consequently, we get the following factor of $w_{4k+1}$: Hence, $s_{4k+1}=R(w_{4(k-1)+1})w_{4(k-1)}^{-1}w_{4(k-1)+3}w_{4(k-1)}^{-1}w_{4(k-1)+1}$ is contained in the prefix $w_{4k+1}$ and it is easy to check that $s_{4k+1}$ is a~palindrome. We have so far found its bilateral extension $1s_{4k+1}0$. Using the fact that $\mathcal{L}(\mathbf{u_p})$ is closed under reversal, it follows that $0s_{4k+1}1 \in \mathcal{L}(\mathbf{u_p})$.
\end{proof}

\begin{example}
Let us write down the shortest two bispecials $s_\ell$ of $\mathbf{u_p}$:
$$s_5=011010110,$$
$$s_7=010101101011001010010101.$$
\end{example}

\begin{proposition}\label{proposition:weakBS}
The factors $s_{2k+1}$ are weak bispecials for all $k\geq 2$.
Moreover, there are no other weak bispecials in the language of ${\mathbf u_p}$ except for $s_{2k+1}$ and their $R$- and $E$-images.
\end{proposition}

Let us postpone the proof of Proposition~\ref{proposition:weakBS} to a~separate subsection since it is long and technical, and provide instead the remaining steps to the proof of Theorem~\ref{thm:counterexample}.

In order to estimate the second difference of complexity, we need to determine the relation of lengths of weak and strong bispecials.

\begin{observation} \label{obs:nerovn}
Consider ${\mathbf u_p} = {\mathbf u}(1^{\omega}, (EERR)^{\omega})$ and $k \in \mathbb{N}, \ k\geq 2$. Then $$|s_{2k+1}| < |w_{2k+1}| < |s_{2k+3}|.$$
\end{observation}

\begin{observation} \label{obs:comp}
Consider ${\mathbf u_p} = {\mathbf u}(1^{\omega}, (EERR)^{\omega})$. Then for all $n \in \mathbb N$ the following holds:
$$\Delta^2\mathcal{C}_{\mathbf{u_p}}(n) \geq \begin{cases}
2 & \text{if} \ n=|w_{2k+1}| \ \text{for some $k\geq 1$}; \\
-2 & \text{if} \ n=|s_{2k+1}| \ \text{for some $k \geq 2$}; \\
0  & \textrm{otherwise}.
\end{cases} $$
\end{observation}
\begin{proof}
We use Equation~\eqref{complexity2diff}.
For $n=|w_{2k+1}|$ we have at least two strong bispecials of $\mathbf{u_p}$: $w_{4\ell+1}$ and $R(w_{4\ell+1})$ (resp. $w_{4\ell+3}$ and $E(w_{4\ell+3})$) by Lemma~\ref{lemma:strongBS}. For  $n =|s_{2k+1}|$ we have exactly two weak bispecials of $\mathbf{u_p}$: $s_{4\ell+1}$ and $E(s_{4\ell+1})$ (resp. $s_{4\ell+3}$ and $R(s_{4\ell+3})$) by Proposition~\ref{proposition:weakBS}. Moreover, Proposition~\ref{proposition:weakBS} states that all other bispecials have at least three bilateral extensions.
\end{proof}

\begin{lemma} \label{lemma:vsechnyfakt}
Consider ${\mathbf u_p} = {\mathbf u}(1^{\omega}, (EERR)^{\omega})$. Let $k \geq 5$. Then $w_k$ contains all factors of $\mathbf{u_p}$ of length less than or equal to $|w_{k-5}|$, except possibly for the images by the antimorphisms $E$ and $R$, and the morphism $ER$. Further on, $w_{k+2}$ contains all factors of $\mathbf{u_p}$ of length less than or equal to $|w_{k-5}|$.
\end{lemma}

\begin{proof}
We will prove the first statement. The second one is its direct consequence -- it suffices to take into account the form of the directive bi-sequence. We will show that $w_s$ for $s \geq k \geq 5$ does not contain (except for $E$-, $R$- and $ER$-images) factors of length less than or equal to $|w_{k-5}|$ other than those ones that are contained in $w_{k}$. To obtain a~contradiction assume that $v$ is the first such factor and that $s$ is the smallest index such that $v$ is contained in $w_s$.
\begin{itemize}
\item
If $s = 4\ell$, then $w_s = w_{s-1}w_{s-4}^{-1}w_{s-1}$. The factor $v$ has to contain the central factor $1w_{s-4}1$ of $w_s$ (otherwise $v$ would be contained already in $w_{s-1}$), which is a~contradiction because $|1w_{s-4}1| > |v|$.
\item
If $s = 4\ell+1$, then $w_s = w_{s-1}10E(w_{s-1})$. By Lemma~\ref{lemma:strongBS}, the factor $w_{s-4}$ is a~central factor of $w_s$ and it has to contain the factor $v$ since $v$ is either a~suffix of $w_{s-1}1$ or a~prefix of $0E(w_{s-1})$ or $v$ contains the central factor $10$ (otherwise, $v$ or $E(v)$ would be contained already in $w_{s-1}$). It is however a~contradiction with the minimality of the index $s$.
\item The remaining two cases are analogous to the above ones.
\end{itemize}
\end{proof}

\begin{corollary}\label{coro:comp1diff}
Consider ${\mathbf u_p} = {\mathbf u}(1^{\omega}, (EERR)^{\omega})$. Then for all $n \geq 10$ the following holds:
$$\begin{array}{rl}
\Delta{\mathcal C}_{\mathbf u_p} (n) \geq 6 & \text{if $|w_{4i+1}|< n \leq |s_{4i+3}|$ or $|w_{4i+3}|< n \leq |s_{4i+5}|$ for some $i \geq 1$};\\
\Delta{\mathcal C}_{\mathbf u_p} (n) \geq 4 & \text{otherwise}.
\end{array}$$

\end{corollary}
\begin{proof}
Let us recall that $\Delta\mathcal{C}(n+1)=\Delta\mathcal{C}(n)+\Delta^2\mathcal{C}(n)$ for all $n \in \mathbb N$. Since $|w_4|=10$, all factors of length $10$ are, according to Lemma~\ref{lemma:vsechnyfakt}, contained in the prefix $w_{11}$ of length $1077$. Checking this prefix by Sage~\cite{St} we determined $\Delta {\mathcal C}(9)=6$. The claim follows then by Observations~\ref{obs:nerovn} and~\ref{obs:comp} taking into account that $|s_5|=9$.
\end{proof}

\begin{proof}[Proof of Theorem~\ref{thm:counterexample}]
In order to get ${\mathcal C}(10)$ it suffices by Lemma~\ref{lemma:vsechnyfakt} to check the prefix $w_{11}$ of length $1077$ because $|w_4|=10$. Using the program Sage~\cite{St} we determined $\mathcal{C}(10) = 42$. It is then a~direct consequence of Corollary~\ref{coro:comp1diff} that ${\mathcal C}(n)>4n$ for all $n \geq 10$.

\end{proof}

\subsection{Proof of Proposition~\ref{proposition:weakBS}}
This section is devoted to quite a~long and technical proof of the fact that the only weak bispecials of $\mathbf{ u_p}$ are $s_{4k+1}$ and $s_{4k+3}$ and their $E$- and $R$-images for all $k \in \mathbb N, \ k \geq 1$.
We will put together several lemmas and observations to get finally the proof.

\begin{lemma}\label{lemma:prefixBS}
Let $v$ be a~prefix of ${\mathbf u_p} = {\mathbf u}(1^{\omega}, (EERR)^{\omega})$. If $v$ is a~bispecial, then $v$ has at least three bilateral extensions and $E(v), R(v), ER(v)$ has at least three bilateral extensions too.
\end{lemma}
\begin{proof}
Denote $a$ the letter for which $va$ is a~prefix of $\mathbf{u_p}$. We can certainly find $k, \ell \in \mathbb N$ such that $va$ is a~prefix of $w_k=R(w_k)$ and $w_\ell=E(w_\ell)$. Then $aR(v)$ is a~suffix of $w_k$ and $\overline{a}E(v)$ is a~suffix of $w_\ell$. By the construction of $\mathbf{u_p}$, the words $aR(v)1$ and $\overline{a}E(v)1$ belong to the language of $\mathbf{u_p}$. Since the language is closed under $E$ and $R$, it follows that $1va$ and $0va$ are factors of $\mathbf{u_p}$ too. Since $v$ is a~bispecial, $v$ has to have a~bilateral extension $bv\overline{a}$ for some $b\in \{0,1\}$. Hence, $v$ has at least three bilateral extensions. The rest of the proof follows by application of the antimorphisms $E, R$ and the morphism $ER$.
\end{proof}

In order to detect all weak bispecial factors, we need to describe all occurrences of $w_k=\vartheta(w_k)$ and $\overline{\vartheta}(w_k)$ and of some of their bilateral extensions. To manage that task, we will distinguish between regular and irregular occurrences.

Let $v$ be a~factor of $\mathbf{u_p}$.
Every element of $\{v, E(v), R(v), ER(v)\}$ is called \emph{an image of} $v$.
Let us define \emph{occurrences (of the images of $v$) generated by a~particular occurrence} $i$ of $v$. Let $k$ be the minimal index such that $w_k$ contains the factor $v$ at the occurrence $i$. Since $w_{k}$ is a~$\vartheta$-palindrome, it contains $\vartheta(v)$ symmetrically with respect to the center of $w_k$. If the corresponding occurrence $j$ of $\vartheta(v)$ is larger than $i$, we say that the occurrence $j$ is generated by the occurrence $i$ of $v$. Assume $w_\ell$ contains occurrences $i_1, \ldots, i_s$ of the images of $v$ generated by the particular occurrence $i$ of $v$. In order to get all occurrences of the images of $v$ generated by the particular occurrence $i$ of $v$ in $w_{\ell+1}$, we proceed in the following way. The prefix $w_{\ell+1}$ is a~$\vartheta$-palindrome for some $\vartheta \in \{E, R\}$, and therefore contains symmetrically with respect to its center occurrences $j_1, \ldots, j_s$ of $v_1, \ldots, v_s$ that are $\vartheta$-images of images of $v$ at the occurrences $i_1, \ldots, i_s$. Putting all occurrences
$i_1, \ldots, i_s, j_1, \ldots j_s$ together, we obtain all occurrences generated by the particular occurrence $i$ of $v$ in $w_{\ell +1}$.

We say that an occurrence of $v$ is \emph{regular} if it is generated by the very first occurrence of any image of $v$ in $\mathbf{u_p}$. Otherwise, we call the occurrence of $v$ \emph{irregular}.

\begin{example}\label{ex:occurrences}
Consider $v=110$. Its images are: $110, 100, 011, 001$.
The first occurrence of an image of $v$ is $i=3$ of $011$ in the palindrome $w_4= 101\underline{011}0101$.
Hence, the occurrence $i=4$ of $v=110$ in $w_4= 1010\underline{110}101$ is regular. (It is the $R$-image of $011$ in $w_4$.)
However, for instance the occurrence $i=9$ of $v=110$ in the $E$-palindrome $w_5=101011010\underline{110}0101001010$ is irregular. (It is not the $E$-image of any image of $v$ at a~regular occurrence in $w_4$.)
\end{example}

\begin{observation} \label{obs:ekvivalence}
Let $v$ be a~factor of $\mathbf{u_p}$. Then $v$ has only regular occurrences in $\mathbf{u_p}$ if and only if any element of $\{v, E(v), R(v), ER(v)\}$ has only regular occurrences in $\mathbf{u_p}$.
\end{observation}

\begin{lemma} \label{lemma:1w41}
Consider ${\mathbf u_p} = {\mathbf u}(1^{\omega}, (EERR)^{\omega})$. Let $k \in \mathbb{N}$. Assume the factors $w_{4k}$ and $w_{4k+2}$ have only regular occurrences in $\mathbf{u_p}$. Then the following statements hold:
\begin{itemize}
\item All irregular occurrences of the factor $1w_{4k}1$ in $\mathbf{u_p}$ are generated by its occurrences as the suffix of the prefix $w_{4\ell}1$ for all $\ell > k$. Moreover, the first regular occurrence of $1w_{4k}1$ is as the central factor of the prefix $w_{4(k+1)}$.
\item All irregular occurrences of the factor $0w_{4k+2}1$ in $\mathbf{u_p}$ are generated by its occurrences as the suffix of the prefix $w_{4\ell+2}1$ for all $\ell > k$. Moreover, the first regular occurrence of $0w_{4k+2}1$ is as the central factor of the prefix $w_{4(k+1)+2}$.
\item All irregular occurrences of the factor $1w_{4k+1}0$ in $\mathbf{u_p}$ are generated by its occurrences as the central factor of the prefix $w_{4\ell+1}$ for all $\ell > k+1$. Moreover, the first regular occurrence of $1w_{4k+1}0$ is as the central factor of the prefix $w_{4(k+1)+1}$.
\item All irregular occurrences of the factor $0w_{4k+3}0$ in $\mathbf{u_p}$ are generated by its occurrences as the central factor of the prefix $w_{4\ell+3}$ for all $\ell > k+1$. Moreover, the first regular occurrence of $0w_{4k+3}0$ is as the central factor of the prefix $w_{4(k+1)+3}$.
\end{itemize}
\end{lemma}

\begin{proof}
We will prove two of the four statements.
\begin{itemize}
\item Let us show the statement for $1w_{4k}1$. The statement for $0w_{4k+2}1$ is an analogy, we leave it thus for the reader. Using Lemma~\ref{lemma:consecutive_members} we know that $w_{4(k+1)} = w_{4k+3}w_{4k}^{-1}w_{4k+3}$. It is easy to see that the bilateral extension of the central factor $w_{4k}$ is $1w_{4k}1$. This bilateral extension occurs in $w_{4(k+1)}$ exactly once. Let us explain why: The factor $w_{4k}$ has only regular occurrences in $\mathbf{u_p}$, therefore $w_{4k+1}$ contains $w_{4k}1$ as prefix and $0E(w_{4k})$ as suffix. Further on, $w_{4k+2}$ contains moreover $0E(w_{4k})1$ and $0w_{4k}1$, and $w_{4k+3}$ contains in addition $1E(w_{4k})0$ and $1w_{4k}0$. Consequently, $0E(w_{4k})0$ is not contained in $w_{4(k+1)}$ and the first occurrence of $1w_{4k}1$ in $w_{4(k+1)}$ is necessarily regular.

Let us study occurrences of $1w_{4k}1$ in the whole word $\mathbf{u_p}$. All regular occurrences of $1w_{4k}1$ are generated by the first occurrence of $1w_{4k}1$ as the central factor of the prefix $w_{4(k+1)}$. We will show that all irregular occurrences of $1w_{4k}1$ are generated by the occurrences of $1w_{4k}1$ as the suffix of the prefix $w_{4\ell}1$ for all $\ell> k$. It is evident that $1w_{4k}1$ is a~suffix of the prefix $w_{4\ell}1$ and the factor $1w_{4k}1$ is here at an irregular occurrence.

For a~contradiction assume that $1w_{4k}1$ occurs at an irregular position that is not generated by the occurrence of $1w_{4k}1$ as the suffix of the prefix $w_{4\ell}1$ for any $\ell> k$. Such an irregular occurrence may as well be generated by an occurrence of $0E(w_{4k})0$. Let $w_s$ be the first prefix that contains such an irregular occurrence of $1w_{4k}1$ (resp. of $0E(w_{4k})0$). Let $m\geq k+1$. If $s = 4m+1$, then $w_s = w_{s-1}10E(w_{s-1})$ and according to Lemma~\ref{lemma:strongBS} the prefix $w_s$ has $1w_{4k+1}0$ as its central factor. The irregular occurrence of $1w_{4k}1$ (resp. of $0E(w_{4k})0$) has to be contained in this factor. But $1w_{4k+1}0$ contains $1w_{4k}1$ only as a~prefix and this occurrence corresponds at the same time to the suffix of $w_{4m}1$, which is a~contradiction. If $s = 4m+2$, then $w_s = w_{s-1}w_{s-4}^{-1}w_{s-1}$ and the irregular occurrence of $1w_{4k}1$ (resp. of $0E(w_{4k})0$) has to contain the central factor of $w_s$: $1w_{s-4}1$. However, $|1w_{s-4}1| > |1w_{4k}1|=|0E(w_{4k})0|$, which is a~contradiction. Let $s = 4m+3$, then $w_s = w_{s-1}(010)^{-1}R(w_{s-1})$. Using Lemma~\ref{lemma:strongBS} the prefix $w_s$ has $w_{4k+3}$ as its central factor. The irregular occurrence of $1w_{4k}1$ (resp. of $0E(w_{4k})0$) has to contain the central factor of $w_s$: $10101$. Consequently, $1w_{4k}1$ (resp. $0E(w_{4k})0$) has to be contained in $w_{4k+3}$, which is a~contradiction.
If $s = 4m+4$, then $w_s = w_{s-1}w_{s-4}^{-1}w_{s-1}$ and the irregular occurrence of $1w_{4k}1$ (resp. of $0E(w_{4k})0$) has to contain the central factor of $w_s$: $1w_{s-4}1$. However, $|1w_{s-4}1| > |1w_{4k}1|=|0E(w_{4k})0|$, which is a~contradiction.
\item Let us show the statement for $1w_{4k+1}0$. The fourth statement is its analogy. The first and thus regular occurrence of $1w_{4k+1}0$ is by Lemma~\ref{lemma:strongBS} and by the assumption on regular occurrences of $w_{4k}$ as the central factor of the prefix $w_{4(k+1)+1}$.

Firstly, let us show that for all $\ell>k$ every occurrence of $1w_{4k+1}0$ (resp. of $0R(w_{4k+1})1$) in the prefixes $w_{4\ell+2}, w_{4\ell+3}, w_{4\ell+4}$ is already generated by an occurrence of an image of $1w_{4k+1}0$ in the prefix $w_{4\ell+1}$.

By Lemma~\ref{lemma:consecutive_members} we can write $w_{4\ell+2}=w_{4\ell+1}w_{4\ell-2}^{-1}w_{4\ell+1}$ and $0w_{4\ell-2}1$ is its central factor. If $w_{4\ell+2}$ contains an occurrence of $1w_{4k+1}0$ (resp. of $0R(w_{4k+1})1$) that is not generated by an occurrence of an image of $1w_{4k+1}0$ in $w_{4\ell+1}$, then this occurrence has to contain $0w_{4\ell-2}1$. This is not possible because for all $\ell >k$ we have
$|0w_{4\ell-2}1| > |1w_{4k+1}0|=|0R(w_{4k+1})1|$. Next, $w_{4\ell+3} = w_{4\ell+2}(010)^{-1}R(w_{4\ell+2})$. The central factor is $10101$ and moreover we know by Lemma~\ref{lemma:strongBS} that $w_{4k+3}$ is also a~central factor of $w_{4\ell+3}$. If $w_{4\ell+3}$ contains an occurrence of $1w_{4k+1}0$ (resp. of $0R(w_{4k+1})1$) that is not generated by an occurrence of an image of $1w_{4k+1}0$ in $w_{4\ell+1}$, then this occurrence has to contain the factor $10101$. Then such an occurrence of $1w_{4k+1}0$ (resp. of $0R(w_{4k+1})1$) is necessarily contained in $w_{4k+3}$. This is not possible since the factor $1w_{4k+1}0$ occurs for the first time in $w_{4(k+1)+1}$ and its $R$-image even later. Finally we have $w_{4\ell+4} = w_{4\ell+3}w_{4\ell}^{-1}w_{4\ell+3}$ and $1w_{4\ell}1$ is its central factor. If $w_{4\ell+4}$ contains an occurrence of $1w_{4k+1}0$ (resp. of $0R(w_{4k+1})1$) that is not generated by an occurrence of an image of $1w_{4k+1}0$ in $w_{4\ell+1}$, then this occurrence has to contain the factor $1w_{4\ell}1$. This is again not possible because of lengths of those factors.

Secondly, let us show that the occurrence of the factor $1w_{4k+1}0$ as the central factor of the prefix $w_{4\ell+1}$ is the only occurrence of $1w_{4k+1}0$ (resp. of $0R(w_{4k+1})1$) in the prefix $w_{4\ell+1}$ that is not generated by any occurrence of an image of $1w_{4k+1}0$ in the prefix $w_{4\ell}$. We have $w_{4\ell+1} = w_{4\ell}10E(w_{4\ell})$. Using Lemma~\ref{lemma:strongBS} it follows that $1w_{4k+1}0$ is the central factor of $w_{4\ell+1}$ and this occurrence is not generated by any image of $1w_{4k+1}0$ contained in $w_{4\ell}$. In order to have another occurrence of $1w_{4k+1}0$ (resp. of $0R(w_{4k+1})1$) in the prefix $w_{4\ell+1}$ so that it is not generated by any image of $1w_{4k+1}0$ in the prefix $w_{4\ell}$, it has to be either a~suffix of $w_{4\ell}1$ or a~prefix of $0E(w_{4\ell})$ or it has to contain the central factor of $w_{4\ell+1}$: $10$. However such an occurrence of $1w_{4k+1}0$ (resp. of $0R(w_{4k+1})1$) has to be contained in the longer central factor of $w_{4\ell+1}$: $w_{4(k+1)+1}$.
This is not possible because $1w_{4k+1}0$ occurs in $w_{4(k+1)+1}$ exactly once as the central factor and this occurrence has been already discussed. Altogether we have described all occurrences of the factor $1w_{4k+1}0$ in $\mathbf{u_p}$. All irregular occurrences of $1w_{4k+1}0$ are thus generated by the occurrences of $1w_{4k+1}0$ as the central factor of $w_{4\ell+1}$, $\ell >k+1$.
\end{itemize}
\end{proof}

\begin{lemma} \label{lemma:vyskyt_w_k}
Consider ${\mathbf u_p} = {\mathbf u}(1^{\omega}, (EERR)^{\omega})$ and $k \in \mathbb N$.
\begin{enumerate}
\item All occurrences of $w_{4k}$ and $E(w_{4k})$ are regular for $k\geq 1$.
\item All occurrences of $w_{4k+2}$ and $R(w_{4k+2})$ are regular.
\item All irregular occurrences of $w_{4k+1}$ and $R(w_{4k+1})$ are generated by the occurrences of $w_{4k+1}$ as the central factor of the prefixes $w_{4\ell+1}$ for all $\ell>k$.
\item All irregular occurrences of $w_{4k+3}$ and $E(w_{4k+3})$ are generated by the occurrences of $w_{4k+3}$ as the central factor of the prefixes $w_{4\ell+3}$ for all $l>k$.
\end{enumerate}
\end{lemma}
\begin{proof}
We will prove only the first and the third statement. The other statements may be proved analogously.
Let us proceed by induction. Assume the first and the third statement hold for some $k \in \mathbb N$.
\begin{enumerate}
\item[1.]
We will first prove that $w_{4(k+1)}$ has only regular occurrences in $\mathbf{u_p}$. Putting together  Lemma~\ref{lemma:1w41}, the induction assumption and the fact that $1w_{4k}1$ and $w_{4(k+1)}$ are both palindromes, it follows that
the occurrence of $w_{4(k+1)}$ is regular if and only if the occurrence of its central factor $1w_{4k}1$ is regular.

Therefore the factor $w_{4(k+1)}$ at an irregular occurrence has to have as its central factor $1w_{4k}1$ at an irregular occurrence, i.e., by Lemma~\ref{lemma:1w41} generated by an occurrence of $1w_{4k}1$ as the suffix of the prefix $w_{4\ell}1$ for some $\ell >k$.
Assume $w_{4(k+1)}$ is at such an occurrence that its central factor $1w_{4k}1$ is the suffix of the prefix $w_{4\ell}1$. By Lemma~\ref{lemma:strongBS} we know that $w_{4\ell+1}=w_{4\ell}10E(w_{4\ell})$ has the central factor $1w_{4k+1}0$. Therefore $w_{4(k+1)}$ having the suffix $1w_{4k}1$ of $w_{4\ell}1$ as its central factor has to contain $1w_{4k+1}0$. This is a~contradiction because using Lemma~\ref{lemma:1w41} and the induction assumption, one can see that the factor $1w_{4k+1}0$ occurs for the first time in $w_{4(k+1)+1}$.

By Observation~\ref{obs:ekvivalence} it follows that $E(w_{4(k+1)})$ has only regular occurrences in $\mathbf{u_p}$ too.

Let us conclude the proof for $k=1$. We will show that $w_4$ and $E(w_4)$ have only regular occurrences in $\mathbf{u_p}$. It is easy to check that $w_8$ contains only regular occurrences of $w_4$ and $E(w_4)$. See Appendix for the form of $w_8$.
Assume $k>8$ and $w_k$ contains the first irregular occurrence of $w_4$ (resp. of $E(w_4)$).
For $k=4m+1$, we have $w_k=w_{k-1}10E(w_{k-1})$. By Lemma~\ref{lemma:strongBS} the factor $w_5$ is a~central factor of $w_k$, hence the irregular occurrence of $w_4$ (resp. of $E(w_4)$) has to be contained in $w_5$, which is a~contradiction. If $k=4m+2$, then $w_k=w_{k-1}w_{k-4}^{-1}w_{k-1}$. The irregular occurrence of $w_4$ (resp. of $E(w_4)$) has to contain the central factor $0w_{k-4}1$, which is a~contradiction.
For $k=4m+3$, we have $w_k=w_{k-1}(010)^{-1}R(w_{k-1})$. The central factor of $w_k$ is $w_7$ by Lemma~\ref{lemma:strongBS}. Therefore the irregular occurrence of $w_4$ (resp. of $E(w_4)$) has to be contained in $w_7$, which is a~contradiction. Finally for $k=4m+4$, the argument is similar as for $k=4m+2$.
Consequently, $w_4$ and $E(w_4)$ have only regular occurrences in $\mathbf{u_p}$.
 \item[3.] We will first prove that all irregular occurrences of $w_{4(k+1)+1}$ are generated by its occurrences as the central factor of the prefixes $w_{4\ell+1}$ for all $\ell >k+1$.
     Since by Lemma~\ref{lemma:1w41} and by the induction assumption, the factor $1w_{4k+1}0$ occurs for the first time as the central factor of the prefix $w_{4(k+1)+1}$ and since both $1w_{4k+1}0$ and $w_{4(k+1)+1}$ are $E$-palindromes and $w_{4(k+1)+1}$ does not contain $0R(w_{4k+1})1$, it follows that the occurrence of $w_{4(k+1)+1}$ is regular if and only if the occurrence of its central factor $1w_{4k+1}0$ is regular.

     We will thus consider irregular occurrences of $1w_{4k+1}0$. We know using Lemma~\ref{lemma:1w41} and the induction assumption that every irregular occurrence of $1w_{4k+1}0$ (resp. of $0R(w_{4k+1})1$) is generated by an occurrence of $1w_{4k+1}0$ as the central factor of $w_{4\ell+1}$ for $\ell>k+1$. It is then a~direct consequence that all irregular occurrences of $w_{4(k+1)+1}$ are generated by the occurrences of $w_{4(k+1)+1}$ as the central factor of the prefixes $w_{4\ell+1}$ for all $\ell>k+1$.

The statement for $R(w_{4(k+1)+1})$ follows using Observation~\ref{obs:ekvivalence}.

It remains to prove the statement for $k=0$.
We have to show that all irregular occurrences of $w_1$ in $\mathbf{u_p}$ are generated by the occurrences of $w_1$ as the central factor of the prefixes $w_{4\ell+1}$ for all $\ell\geq 1$. It is easy to show that the first irregular occurrence of $w_1$ is the occurrence as the central factor of the prefix $w_5$. Let $m > 5$ and let $w_m$ contain the first irregular occurrence of $w_1$ (resp. of $R(w_1)$) that is not generated by the occurrence of $w_1$ as the central factor of the prefix $w_5$. If $m = 4\ell+2$, then $w_m = w_{m-1}w_{m-4}^{-1}w_{m-1}$. Then the irregular occurrence of $w_1$ (resp. of $R(w_1)$) has to contain the central factor of $w_m$: $0w_{m-4}1$, which is not possible. If $m = 4\ell+3$, then $w_m = w_{m-1}(010)^{-1}R(w_{m-1})$. By Lemma~\ref{lemma:strongBS} the factor $w_3$ is a~central factor of $w_m$. Then the irregular occurrence of $w_1$ (resp. of $R(w_1)$) has to be contained in $w_3$, which is a~contradiction. If $m = 4\ell+4$, then $w_m = w_{m-1}w_{m-4}^{-1}w_{m-1}$. Then the irregular occurrence of $w_1$ (resp. of $R(w_1)$) has to contain the central factor of $w_m$: $1w_{m-4}1$, which is not possible. If $m = 4\ell+5$, then $w_m = w_{m-1}10E(w_{m-1})$. By Lemma~\ref{lemma:strongBS} the factor $w_5$ is a~central factor of $w_m$. Then the irregular occurrence of $w_1$ (resp. of $R(w_1)$) has to be contained in $w_5$. It follows that $w_1$ has to be the central factor of $w_{4\ell+5}, \ \ell \geq 1$.

\end{enumerate}

\end{proof}

In the proof of the last and essential lemma, we will make use of the following observation.
\begin{observation}\label{obs:p_k}
Consider ${\mathbf u_p} = {\mathbf u}(1^{\omega}, (EERR)^{\omega})$.
For all $k\in \mathbb N, \ k\geq 1$, let:
$$p_{4k+1} = w_{4(k-1)+3}w_{4(k-1)}^{-1}w_{4(k-1)+1},$$
 $$p_{4k+3} = w_{4k+1}w_{4k-2}^{-1}w_{4(k-1)+3}.$$
 Then the factor $p_{4k+1}$ is a~suffix of $s_{4k+1}$ and a~prefix of $w_{4k}$
 and similarly the factor $p_{4k+3}$ is a~suffix of $s_{4k+3}$ and a~prefix of $w_{4k+2}$.
\end{observation}

\begin{lemma}\label{lemma:nonprefixBS}
Let $v$ be a~factor, but not an image of a~prefix of ${\mathbf u_p} = {\mathbf u}(1^{\omega}, (EERR)^{\omega})$.
The following statements hold:
\begin{itemize}
\item If $v$ is neither an $E$-palindrome, nor a~palindrome, then $v$ is not bispecial in $\mathbf u_p$.
\item If $v$ is an $E$-palindrome or a~palindrome, but different from $s_{4k+1}$, $E(s_{4k+1})$, $s_{4k+3}$, $R(s_{4k+3})$ for all $k\geq 1$, then $v$ is either not a~bispecial, or it is a~bispecial with three bilateral extensions.
\item If $v$ is equal to one of the bispecials $s_{4k+1}, \ E(s_{4k+1}), \ s_{4k+3}, \ R(s_{4k+3})$ for some $k\geq 1$, then $v$ is a~weak bispecial.
\end{itemize}
\end{lemma}
\begin{proof}
We will find the minimal index $k$ such that $w_k$ contains an image of the factor $v$. Let the first occurrence of an image of $v$ correspond without loss of generality to $v$. Let us discuss the possible cases.

No such factors are contained in $w_3$.

\begin{enumerate}
\item Let $k=4\ell, \ \ell \geq 1$. We have $w_{4\ell}=w_{4(\ell-1)+3}w_{4\ell-4}^{-1}w_{4(\ell-1)+3}$. The bilateral extension of the central factor is $1w_{4\ell-4}1$. According to Lemma~\ref{lemma:1w41} the factor $1w_{4\ell-4}1$ occurs for the first time as the central factor of $w_{4\ell}$ and moreover $w_{4\ell}$ does not contain $0E(w_{4\ell-4})0$. Since $v$ is not contained in $w_{4(\ell-1)+3}$, the factor $v$ has to contain the central factor of $w_{4\ell}$: $1w_{4\ell-4}1$. Hence $v$ occurs once in $w_{4\ell}$. Let us denote $avb$ the corresponding bilateral extension of $v$, where $a,b \in \{0,1\}$. If $v$ is a~$\vartheta$-palindrome, then thanks to the unique occurrence of the palindrome $1w_{4\ell-4}1$ and the absence of $0E(w_{4\ell-4})0$ in $w_{4\ell}$, the factor $v$ has to be the central palindrome of $w_{4\ell}$. In this case its bilateral extension is $ava$, i.e. $a=b$. Moreover, $v$ is distinct from $s_{4m+1}$ and $E(s_{4m+1})$ for all $m \in \mathbb N$ because these palindromes do not have the central factor $1w_{4\ell-4}1$.

    We will now study irregular occurrences of $v$.
    It is not difficult to see that regular occurrences of $1w_{4\ell-4}1$ are factors of regular occurrences of $v$ and $R(v)$. Therefore we have to look at irregular occurrences of $1w_{4\ell-4}1$.
    The first such occurrence is as the suffix of the prefix $w_{4\ell}1$. Then $v$ cannot contain the central factor $1w_{4(l-1)+1}0$ of $w_{4\ell+1}=w_{4\ell}10E(w_{4\ell})$ since by Lemmas~\ref{lemma:1w41} and~\ref{lemma:vyskyt_w_k} the factor $1w_{4(l-1)+1}0$ occurs for the first time in $w_{4\ell+1}$ while $v$ occurs already in $w_{4\ell}$. Consequently and since $v$ contains $1w_{4\ell-4}1$ once, $v$ has to be contained in the suffix of $s_{4\ell+1}$: $p_{4\ell+1}$ defined in Observation~\ref{obs:p_k}. If $v$ is not a~suffix of $p_{4\ell+1}$, then since $p_{4\ell+1}$ is a~prefix of $w_{4\ell}$, we do not get any new bilateral extension of $v$. If $v$ is a~suffix of $p_{4\ell+1}$ and thus of the bispecial $s_{4\ell+1}$, the word $\mathbf{u_p}$ contains the bilateral extension $av\overline{b}$ too. All other irregular occurrences of $1w_{4\ell-4}1$ are generated by its occurrences as the suffix of the prefix $w_{4m}1, \ m>\ell$. It is not difficult to see that such occurrences do not provide any new bilateral extension of~$v$. Altogether we have found for $v$ that is not an $R$ palindrome the bilateral extension $avb$ and possibly $av\overline{b}$. Thus, such a~factor $v$ is not bispecial. If $v$ is a~palindrome, then its bilateral extension is either only $ava$ and it is not a~bispecial, or its bilateral extensions are $ava, av\overline{a}$ and by the fact that the language is closed under reversal also $\overline{a}va$. Therefore $v$ is a~bispecial with three bilateral extensions.
\item Let $k=4\ell+1, \ \ell \geq 1$. The factor $v$ occurs for the first time in $w_{4\ell+1}=w_{4\ell}10E(w_{4\ell})$.
Therefore $v$ contains either the central factor $1w_{4(\ell-1)+1}0$ of $w_{4\ell+1}$ or $v$ has to contain at least the suffix $1w_{4(\ell-1)+3}1$ of $w_{4\ell}1$. If $v$ does not contain neither $1w_{4(\ell-1)+1}0$ nor $1w_{4(\ell-1)+3}1$, then $v$ itself is contained in $p_{4\ell+1}$ defined in Observation~\ref{obs:p_k}. However, $p_{4\ell+1}$ is a~prefix of $w_{4\ell}$, therefore $v$ would be contained already in $w_{4\ell}$.

If $v$ contains $1w_{4(\ell-1)+1}0$, then $v$ occurs in $w_{4\ell+1}$ once since $w_{4\ell+1}$ contains
$1w_{4(\ell-1)+1}0$ only once by Lemmas~\ref{lemma:1w41} and~\ref{lemma:vyskyt_w_k}. If $v$ contains $1w_{4(\ell-1)+3}1$, then $v$ occurs in $w_{4\ell+1}$ only once too, as one can easily check using Lemma~\ref{lemma:vyskyt_w_k}. Let $avb$ denote the corresponding bilateral extension of $v$. If $v$ is a~$\vartheta$-palindrome, then two cases are possible:
If $v$ contains $1w_{4(\ell-1)+1}0$, then $v$ is an $E$-palindromic central factor of $w_{4\ell+1}$. Then $v$ is not equal to $s_{4m+3}$ or $R(s_{4m+3})$ for any $m\in \mathbb N$ because neither $s_{4m+3}$ nor $R(s_{4m+3})$ is a~central factor of $w_{4\ell+1}$. The bilateral extension of $v$ is then $av\overline{a}$. If $v$ contains $1w_{4(\ell-1)+3}1$, then $v$ has to be a~palindrome with the central factor $1w_{4(\ell-1)+3}1$. The longest such palindrome in $w_{4\ell+1}$ is $s_{4\ell+1}$. If $v=s_{4\ell+1}$, then its bilateral extension is $av\overline{a}$. If $v$ is a~shorter palindrome,  i.e., a~central factor of $s_{4\ell+1}$, then its bilateral extension is $ava$.

%

It is not difficult to see that no irregular occurrence of any image of $v$ is contained in $w_{4\ell+1}$. Consider irregular occurrences of images of $v$ first in $w_{4\ell+2}=w_{4\ell+1}w_{4\ell-2}^{-1}w_{4\ell+1}$. The first irregular occurrence of an image of $v$ has to contain the central factor $0w_{4\ell-2}1$ of $w_{4\ell+2}$, which is not possible because $0w_{4\ell-2}1$ occurs by Lemmas~\ref{lemma:1w41} and~\ref{lemma:vyskyt_w_k} for the first time in $w_{4\ell+2}$, and moreover $w_{4\ell+2}$ does not contain $1R(w_{4\ell-2})0$. Thus $w_{4\ell+2}$ does not contain any irregular occurrence of any image of $v$. Similarly, no irregular occurrence of an image of $v$ is contained in $w_{4\ell+3}=w_{4\ell+2}(010)^{-1}R(w_{4\ell+2})$. The image of $v$ cannot contain the central factor $0w_{4(\ell-1)+3}0$ because this factor occurs for the first time in $w_{4\ell+3}$ and its $E$-image even later. The image of $v$ cannot contain the suffix $0w_{4\ell-2}1$ of $w_{4\ell+2}1$ because $0w_{4\ell-2}1$ occurs for the first time in $w_{4\ell+2}$ and its $R$-image even later. This implies however that the image of $v$ is contained in $p_{4\ell+3}$ defined in Observation~\ref{obs:p_k}. And since $p_{4\ell+3}$ is a~prefix of $w_{4\ell+2}$, such occurrence of the image of $v$ is regular.
Consider an image of $v$ has an irregular occurrence in $w_{4\ell+4}$. Then the image of $v$ has to contain its central factor $1w_{4\ell}1$, which occurs however for the first time in $w_{4\ell+4}$ and its $E$-image even later. Therefore it is not possible.
No new irregular occurrences can appear in larger prefixes: for $s>\ell$, the prefixes $w_{4s+2}$ and $w_{4s+4}$ has too long central factors that $v$ has to contain, while $w_{4s+1}$ and $w_{4s+3}$ have central factors $w_{4\ell+1}$ (resp. $w_{4\ell+3}$) and these cases have been already discussed.

If $v$ is not a~$\vartheta$-palindrome, then the only bilateral extension of $v$ is $avb$, thus $v$ is not a~bispecial. If $v$ is an $E$-palindrome, then its only bilateral extension is $av\overline{a}$ and we do not get any new bilateral extension by application of $E$. Hence $v$ is not a~bispecial. If $v$ is a~palindrome, but distinct from $s_{4\ell+1}$, then its bilateral extension is $ava$ and we do not get any new bilateral extension by application of $R$. Finally, if $v=s_{4\ell+1}$, then its bilateral extension is $av\overline{a}$ and by application of $R$ we get $\overline{a}va$, thus $v$ is a~weak bispecial.
\item Let $k=4\ell+2$. This case is analogous to the first one.
\item Let $k=4\ell+3$. This case is analogous to the second one.
\end{enumerate}
\end{proof}

\begin{proof}[Proof of Proposition~\ref{proposition:weakBS}]
The result follows from Lemmas~\ref{lemma:prefixBS} and~\ref{lemma:nonprefixBS}.
\end{proof}

\subsection{Complexity of $\mathbf u_p$ is significantly larger than $4n$}
Let us prove that the infinite word $\mathbf u_p$ from the counterexample to Conjecture $4n$ not only satisfies ${\mathcal C}_{\mathbf u_p}(n)> 4n$ for all $n \geq 10$, but its complexity is significantly larger than $4n$.

We start with a~simple observation concerning the lengths of weak bispecial factors.
\begin{observation}\label{obs:length_weakBS}
For all $i \in \mathbb N, \ i\geq 1$, we have:
$$\begin{array}{rclcl}
|s_{4i+5}|-|w_{4i+3}|&=&2|w_{4i+1}|-2|w_{4i}|&=&2|w_{4i}|+4,\\
|s_{4i+3}|-|w_{4i+1}|&=&2|w_{4i-1}|-2|w_{4i-2}|&=&2|w_{4i-2}|-6.
\end{array}$$
\end{observation}
\begin{proof}
It follows from the form of weak bispecials described in Lemma~\ref{lemma:BSs_k} that
\begin{equation}\label{eq:weakBS1}
\begin{array}{rcl}
|s_{4i+5}|&=&2|w_{4i+1}|+|w_{4i+3}|-2|w_{4i}|,\\
|s_{4i+3}|&=&2|w_{4i-1}|+|w_{4i+1}|-2|w_{4i-2}|.
\end{array}
\end{equation}
Applying then Lemma~\ref{lemma:consecutive_members}, we obtain the statement.
\end{proof}
\begin{theorem}\label{limsup}
Let ${\mathbf u_p} = {\mathbf u}(1^{\omega}, (EERR)^{\omega})$. Then its complexity satisfies:
$$\limsup\frac{{\mathcal C}_{\mathbf u_p}(n)}{n} \geq 4.57735.$$
\end{theorem}
\begin{proof}
For $n\geq 10$ we have ${\mathcal C}(n)={\mathcal C}(10)+\sum_{\ell=10}^{n-1}\Delta {\mathcal C}(\ell)$,
where by Corollary~\ref{coro:comp1diff} we have for $\ell \geq 10$:
$$\begin{array}{rl}
\Delta{\mathcal C}(\ell) \geq 6 & \text{if $|w_{4i+1}|<\ell \leq |s_{4i+3}|$ or $|w_{4i+3}|<\ell \leq |s_{4i+5}|$ for some $i \geq 1$};\\
\Delta{\mathcal C}(\ell) \geq 4 & \text{otherwise}.
\end{array}$$

Let us set $n=|s_{4k+5}|+1$. Since ${\mathcal C}(10)=42$, we get the following expression:
$${\mathcal C}(n)\geq 42+4(n-10)+2\sum_{i=1}^k\left(|s_{4i+3}|-|w_{4i+1}|\right)+2\sum_{i=1}^k\left(|s_{4i+5}|-|w_{4i+3}|\right).$$
Inserting formulas from Observation~\ref{obs:length_weakBS}, we obtain:
\begin{equation}\label{eq:estimateC1}
{\mathcal C}(n)\geq 4n+2-4k+4\sum_{i=1}^k\left(|w_{4i-2}|+|w_{4i}|\right).
\end{equation}
Using~\eqref{eq:weakBS1} we have $|s_{4k+5}|=2|w_{4k+1}|+|w_{4k+3}|-2|w_{4k}|$ and applying Lemma~\ref{lemma:consecutive_members} we obtain
\begin{equation}\label{s4k+5}
|s_{4k+5}|=2|w_{4k}|+2|w_{4k+2}|+1.
\end{equation}
Therefore it suffices to deal with $w_{n}$ for even $n$.

By Lemma~\ref{lemma:consecutive_members} we easily deduce the following set of equations:
\begin{equation}\label{eq:set_for_w_2n}
\begin{array}{rcl}
|w_{4k+2}|&=&4|w_{4k}|-|w_{4k-2}|+4,\\
|w_{4k+4}|&=&4|w_{4k+2}|-|w_{4k}|-6,\\
|w_{4k}|&=&4|w_{4k-2}|-|w_{4k-4}|-6.
\end{array}
\end{equation}
We multiply the first equation by four and add the remaining two equations so that we get the following recurrence equation for $|w_{4k}|$:
$$|w_{4k+4}|=14|w_{4k}|-|w_{4k-4}|+4.$$
The initial values are $|w_0|=0$ and $|w_4|=10$.
The solution of the above recurrence equation reads
\begin{equation}\label{eq:solution}
|w_{4k}|=\frac{1+2\sqrt{3}}{6}\tau^k+\frac{1-2\sqrt{3}}{6}{(\tau')}^k-\frac{1}{3},
\end{equation}
where $\tau=7+4\sqrt{3}>1$ and $\tau'=7-4\sqrt{3} \in (0,1)$ are roots of the equation $x^2-14x+1=0$.
It follows using the last equation from~\eqref{eq:set_for_w_2n} that
\begin{equation}\label{eq:w_{4k-2}}
4|w_{4k-2}|=|w_{4k-4}|+|w_{4k}|+6=\frac{16}{3}+\frac{1+2\sqrt{3}}{6}
\tau^k(\tau'+1)+\frac{1-2\sqrt{3}}{6}{(\tau')}^k(\tau+1).\end{equation}
Consequently,
$$4\sum_{i=1}^k\left(|w_{4i-2}|+|w_{4i}|\right)=
4k+\frac{1+2\sqrt{3}}{6}(\tau'+5)\frac{\tau}{\tau-1}(\tau^k-1)+c_1(k),$$
where $c_1(k)$ is a~bounded sequence.

Inserting the previous formula into~\eqref{eq:estimateC1}, we have:
\begin{equation}\label{eq:complexity2}
{\mathcal C}(n)\geq 4n+2+\frac{1+2\sqrt{3}}{6}(\tau'+5)\frac{\tau}{\tau-1}(\tau^k-1)+c_1(k).
\end{equation}

In order to continue with the estimate on complexity, we have to express the relation between $k$ and $n$. We know that $n=|s_{4k+5}|+1$. Combining Equations~\eqref{s4k+5}, \eqref{eq:solution} and~\eqref{eq:w_{4k-2}}, we obtain:
$$|s_{4k+5}|=3+\frac{1+2\sqrt{3}}{6}\frac{5+\tau}{2}\tau^k+c_2(k),$$
where $c_2(k)$ is again a~bounded sequence.
Consequently,
$$n =|s_{4k+5}|+1=\frac{1+2\sqrt{3}}{6}\frac{5+\tau}{2}\tau^k+c_3(k),$$
where constants have been included in the bounded sequence $c_3(k)$.
Hence, we have:
$$\tau^k=\frac{6}{1+2\sqrt{3}}\frac{2}{5+\tau}n+c_4(k),$$
where $c_4(k)$ is a~bounded sequence.
Inserting the previous formula into~\eqref{eq:complexity2}, we obtain the following lower bound on complexity:
$${\mathcal C}(n)\geq 4n+n\frac{\tau'+5}{\tau+5}\frac{2\tau}{\tau-1}+c_5(k),$$
where $c_5(k)$ is a~bounded sequence, where all constants have been included.
Finally, we obtain:
$$\limsup\frac{{\mathcal C}(n)}{n}\geq 4+\frac{\tau'+5}{\tau+5}\frac{2\tau}{\tau-1}\doteq 4,57735.$$

\end{proof}

\section{Open problems}\label{sec:open_problems}
It remains as an open problem to determine a~new upper bound on the complexity of binary generalized pseudostandard words.

\begin{itemize}
\item Let us start here with a~simple observation: If both $E$ and $R$ occur in the sequence $\Theta$ an infinite number of times, then since the language of $\mathbf u(\Delta, \Theta)$ is closed under the antimorphisms $E$ and $R$, the first difference of complexity has even values, i.e., $\Delta {\mathcal C}(n) \in \{2,4,6,\ldots \}$.
\item It might be helpful to illustrate what is the situation for sufficiently long left special factors of the Thue--Morse word and what seems to be the situation for the word $\mathbf u_p$. See Figure~\ref{fig:left_specials}.
In both cases, there are two infinite left special branches -- the infinite word itself and its $ER$-image, i.e., the word that arises when exchanging ones with zeroes. The weak bispecial factors form finite branches and the common prefixes of the weak bispecials and the infinite left special branch correspond to strong bispecials. The first difference of complexity $\Delta {\mathcal C}(n)$ equals to the number of left special factors of length $n$. In the case of $\mathbf u_p$ in contrast to $\mathbf u_{TM}$, the detached parts of finite branches may overlap, thus the first difference of complexity is larger. The question is whether it is possible to construct an example where there are even more overlapping detached parts of finite branches.

\begin{figure}[!h]
\centering
\def\svgwidth{\columnwidth}
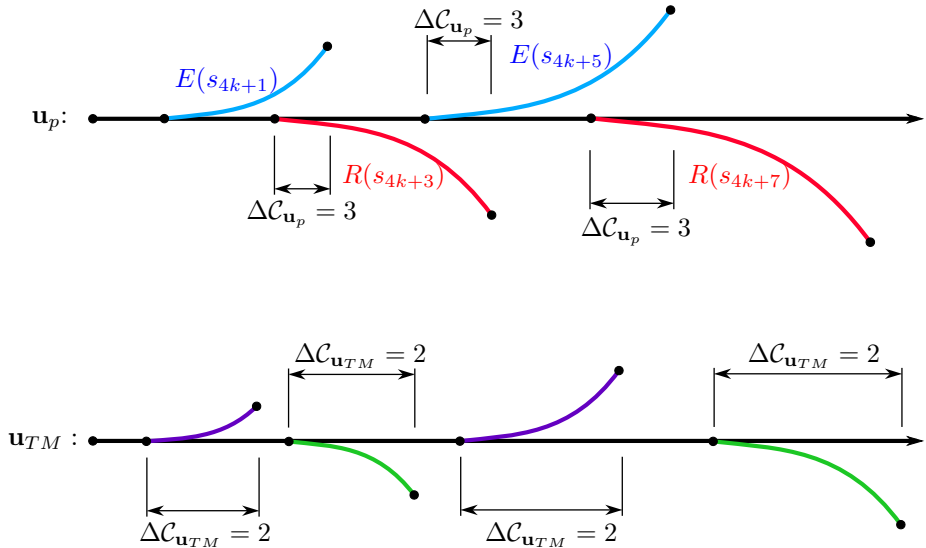
\caption{Infinite left special factors of $\mathbf u_p$ and $\mathbf u_{TM}.$ We illustrate in each case only one of two infinite left special branches. Hence, the total first difference of complexity has to be doubled.}\label{fig:left_specials}
\end{figure}

\item Let us state a~new conjecture based on our computer experiments:
\begin{conjecture}[Conjecture $6n$]
Let $\mathbf u$ be a~binary generalized pseudostandard word, then its complexity satisfies $${\mathcal C}_{\mathbf u}(n)<6n \quad \text{for all $n \in \mathbb N$}.$$
\end{conjecture}
The arguments supporting our conjecture are as follows:
     \begin{enumerate}
     \item On one hand, in all our examples the first difference of complexity satisfies $\Delta {\mathcal C}(n) \leq 6$.
     \item On the other hand, we have checked for $\mathbf u=\mathbf u(1^{\omega}, (RRRRREEEEE)^{\omega})$ that ${\mathcal C}_{\mathbf u}(n)>5n$ for some $n \in \mathbb N$.
         \end{enumerate}
\end{itemize}

\section*{Acknowledgements}
We acknowledge the financial support of Czech Science
Foundation GA\v{C}R 13-03538S and the grant SGS14/205/OHK4/3T/14.
We would like to thank Edita Pelantov\'a for her useful comments and advice concerning in particular complexity of infinite words.

%
%
%

\section{Appendix}\label{sec:Appendix}
In this appendix, we list the members of the sequence $(w_k)_{k=1}^{9}$ for the infinite word $\mathbf{u_p} = {\mathbf u}(1^{\omega}, (EERR)^{\omega})$. For their generation the program Sage~\cite{St} was used.
We have moreover highlighted the first occurrences of weak bispecials $s_{2k+1}, \ k\geq 2$.
\begin{align}
    w_1 =& \;10 \nonumber \\[1mm]
    w_2 =& \;1010 \nonumber \\[1mm]
    w_3 =& \;10101 \nonumber \\[1mm]
    w_4 =& \;1010110101 \nonumber \\[1mm]
    w_5 =& \;101\underbrace{011010110}_{s_5}0101001010 \nonumber \\[1mm]
    w_6 =& \;1010110101100101001010110101100101001010 \nonumber \\[1mm]
    w_7 =& \;10101101011001010\underbrace{010101101011001010010101}_{s_7}001010011010 \nonumber \\
        &  \;110101001010011010110101 \nonumber \\[1mm]
    w_8 =& \;10101101011001010010101101011001010010101001010011010 \nonumber \\
        &  \;11010100101001101011010110010100101011010110010100101 \nonumber \\
        &  \;01001010011010110101001010011010110101 \nonumber \\[1mm]
         w_9 =& \;10101101011001010010101101011001010010101001010011010 \nonumber \\
    		&  \;11\underbrace{010100101001101011010110010100101011010110010100101} \nonumber \\
    		&  \;\underbrace{01001010011010110101001010011010110101100101001010}_{s_9}011 \nonumber \\
		&  \;01011010100101001101011010101101011001010010101101011 \nonumber \\
		&  \;00101001010011010110101001010011010110101011010110010 \nonumber \\
    		&  \;1001010110101100101001010 \nonumber
        \end{align}

\end{document}